\documentclass[a4paper,11pt, oneside]{amsart}

\usepackage[cp1250]{inputenc}
\usepackage{amsmath}
\usepackage{amssymb,amsbsy,amsmath,amsfonts,amssymb,amscd}
\usepackage{latexsym}
\usepackage{amsthm}
\usepackage{mathrsfs}
\usepackage{graphics}
\usepackage{color}
\usepackage{tikz}
\usetikzlibrary{arrows}
\usepackage{tikz-cd}
\usepackage{pdfpages} %cover
\usepackage{longtable}
\usepackage{xy}
\usepackage{chngcntr}
\usepackage{array}
\usepackage{color}
\usepackage{comment}
\usepackage{enumitem}
\usepackage{geometry}
\usepackage{setspace}
\usepackage{multirow}
\usepackage[pagebackref=true]{hyperref}
\usepackage{textcmds} % for english "

\usepackage[alphabetic,backrefs,lite]{amsrefs} % for bibliography

\input xy
\xyoption{all}

\newcommand\sO{{\mathcal O}}
\newcommand\sC{{\mathcal C}}

\newcommand\sN{{\mathcal N}}
\newcommand\sK{{\mathcal K}}

\newcommand{\CC}{\ensuremath{\mathbb{C}}}
\newcommand{\RR}{\ensuremath{\mathbb{R}}}
\newcommand{\ZZ}{\ensuremath{\mathbb{Z}}}

\newcommand{\PP}{\ensuremath{\mathbb{P}}}

\newcommand\la{\lambda}
\newcommand\Lam{\Lambda}

\newcommand\ze{\zeta}

\DeclareMathOperator{\Fix}{Fix}

\DeclareMathOperator{\Stab}{Stab}

\DeclareMathOperator{\PGL}{PGL}
\DeclareMathOperator{\AGL}{AGL}

\DeclareMathOperator{\Aut}{Aut}
\DeclareMathOperator{\End}{End}

\DeclareMathOperator{\Sing}{Sing}
\DeclareMathOperator{\diag}{diag}

\DeclareMathOperator{\Jac}{Jac}
\DeclareMathOperator{\id}{id}
\DeclareMathOperator{\Heis}{He}
\DeclareMathOperator{\He}{He}

\DeclareMathOperator{\im}{im}

\DeclareMathOperator{\GL}{GL}

\DeclareMathOperator{\Bihol}{Bihol}
\DeclareMathOperator{\Deck}{Deck}
\DeclareMathOperator{\SL}{SL}

\def\eea{\end{eqnarray*}}
\def\bea{\begin{eqnarray*}}

\newcommand\dual{\mathrel{\raise3pt\hbox{$\underline{\mathrm{\thinspace d
\thinspace}}$}}}
\newcommand\qe{\ifhmode\unskip\nobreak\fi\quad $\Box$}       % box for QED

\def\BOX{\hfill\lower.5\baselineskip\hbox{$\Box$}}
\newcommand{\bigslant}[2]{{\raisebox{.2em}{$#1$}/\raisebox{-.2em}{$#2$}}}

\DeclareMathOperator{\SheafHom}{\mathscr{H}\text{\kern -4pt {\textit{om}}}\,}
\DeclareMathOperator{\SheafExt}{\mathscr{E}\text{\kern -3pt {\textit{xt}}}\,}

\newtheorem{theorem}{Theorem}
\newtheorem{theo}[theorem]{Theorem}
\newtheorem{remark}[theorem]{Remark}
\newenvironment{rem}{\begin{remark}\rm}{\end{remark}}

\newtheorem{prop}[theorem]{Proposition}
\newtheorem{cor}[theorem]{Corollary}
\newtheorem{lemma}[theorem]{Lemma}
\newtheorem{example}[theorem]{Example}

\numberwithin{theorem}{section}
\numberwithin{equation}{section}

\theoremstyle{definition}
\newtheorem{defin}[theorem]{Definition}
\newenvironment{definition}{\begin{defin}\rm}{\end{defin}}

\makeatletter
\def\tagform@#1{\maketag@@@{\ignorespaces#1\unskip\@@italiccorr}}
\makeatother

\newcolumntype{H}{@{}>{\lrbox0}l<{\endlrbox}} %Spalten ignorieren
\newcommand{\mylabel}[2]{#2\def\@currentlabel{#2}\label{#1}}

\hypersetup{colorlinks=true, unicode=true, linkcolor=[rgb]{0.10,0.05,0.67}, citecolor=[rgb]{0.10,0.05,0.67}, filecolor=[rgb]{0.10,0.05,0.67}, urlcolor=[rgb]{0.10,0.05,0.67}}

\begin{document}

\title[Crystallographic Groups and Calabi-Yau 3-folds]{Crystallographic Groups and Calabi-Yau 3-folds of Type $\mathrm{III}_0$}
\author{ Christian Glei\ss ner and Julia Kotonski}
\address{Christian Gleissner and Julia Kotonski  \newline University of Bayreuth, Universit\"atsstr. 30, D-95447 Bayreuth, Germany}
\email{christian.gleissner@uni-bayreuth.de,  julia.kotonski@uni-bayreuth.de}

\thanks{
\textit{2020 Mathematics Subject Classification.} Primary: 14J30, 14J32, 14L30,  Secondary: 20C15, 20H15, 32Q25.\\
\textit{Keywords}: Calabi-Yau three-folds, crystallographic groups, torus quotients \\
\textit{Acknowledgements:} The authors would like to thank Ingrid Bauer, Fabrizio Catanese and Andreas Demleitner
 for useful comments and discussions. The second author is supported by the \qq{Studienstiftung des deutschen Volkes}.
}

\begin{abstract}
We provide a fine classification of Gorenstein quotients of three-dimensional abelian varieties with isolated singularities, up to biholomorphism and homeomorphism. This refines a result of Oguiso and Sakurai about fibred Calabi-Yau threefolds of type $\mathrm{III}_0$. 
Our proof relies on methods of crystallographic group theory applied to the orbifold fundamental groups of such quotients.
\end{abstract}

\maketitle

\tableofcontents

\section{Introduction}

A \textit{Calabi-Yau threefold} $X$ is a 
% not necessarily simply connected
  complex 
three-dimensional $\mathbb Q$-factorial variety with at most terminal singularities,  trivial canonical divisor
and vanishing irregularity $q_1:=h^1(\mathcal O_X)$. 
 A \textit{fibred} Calabi-Yau threefold is a pair $(X,f)$, where $X$ is a Calabi-Yau threefold and $f\colon X \to Z$ is a contraction, i.e., a surjective holomorphic map with connected fibers onto a normal, projective variety $Z$ of positive dimension. Such a contraction $f$ is defined by the complete linear system of a nef divisor $D$, which has,  according to Bogomolov-Miyaoka-Yau,  
 non-negative intersection 
 $c_2(X)\cdot D\geq 0$.  In \cite{OguisoFiberSpace} and \cite{OguisoQuasiProduct}, Oguiso studied fibred Calabi-Yau threefolds by  subdividing  them 
 into six classes,  
 according to the three possible values of the Iitaka dimension of $D$  and the property, whether the intersection number  $c_2(X)\cdot D$ is zero or strictly positive.
In our article, we are interested in the class, where $D$ is big, i.e., the Iitaka dimension of $D$ is three, and $c_2(X)\cdot D=0$. 
In the literature, they are usually labeled  as 
type $\mathrm{III}_0$. Oguiso characterized them as the examples where the  base  $Z$ of the contraction $f\colon X \to Z$ is three-dimensional, has trivial canonical divisor $K_Z$ and vanishing second Chern class. By a 
famous result of Shepherd-Barron and Wilson \cite{ShepherdWilson}, this amounts to saying  that   $Z$ is biholomorphic to a quotient of an abelian threefold $A$ by a finite translation-free group $G$ which acts freely in codimension two, has at least one fixed point and preserves the volume form of $A$.
%Since $W$ has trivial canonical divisor, the  action must preserve the volume form of the torus. 
Based on this result, Oguiso and Sakurai proved that each fibered Calabi-Yau threefold $(X,f)$ of type $\mathrm{III}_0$ is isomorphic (as a pair) to the necessarily unique crepant resolution 
 $f'\colon X'\to Z$ of the quotient $Z=A/G$ (cf. \cite{OguisoSimplyConnected}, \cite{OguisoQuotientType}). 
 In particular, the pair $(X,f)$  is,  up to isomorphism,  determined by $Z=A/G$.
 % Hence, it is enough to classify these quotients.
A fine classification is only established if the quotient  is simply connected  \cite{RoanYau},  \cite{RoanKummer}, \cite{OguisoSimplyConnected}:  there exists exactly two 
biholomorphism classes of quotients  $A/G$ which are represented by  
\begin{itemize}
 \item[(a)] $A:=E^3$, the product of three Fermat elliptic curves $E:=\CC/\ZZ[\ze_3]$, and $G:=\langle \zeta_3\cdot \id\rangle$,  
 \item[(b)]  $A:=\Jac(Q)$, the Jacobian variety of Klein's plane quartic curve 
 \[
 Q:=\lbrace x_0 x_1^3+x_1 x_2^3+x_2 x_0^3 = 0  \rbrace \subset \mathbb P^2_{\mathbb C}, 
 \]
 and $G:=\langle \diag(\ze_7,\ze_7^2,\ze_7^4)\rangle$.   
\end{itemize}
We point out that the abelian threefold $\Jac(Q)$ in (b) has many other descriptions. Indeed, by a result of Shimura-Taniyama 
 \cite{Shimura},  any abelian threefold  with automorphism $\diag(\ze_7,\ze_7^2,\ze_7^4)$ is biholomorphic to $\Jac(Q)$. 
  
If $A/G$ has non-trivial fundamental group, two additional groups occur, namely $\ZZ_3^2$ and the Heisenberg group $\He(3)$ of order $27$. Here, $A$ is always isomorphic to the product of three Fermat elliptic curves $E^3$. However, in contrast to the simply connected case, for each of these  groups $G=\ZZ_3^2$ and $\He(3)$, there are several different actions leading to different 
biholomorphism and even homeomorphism 
classes of quotients  $A/G$. Our main result is  a precise list of all of them.

\begin{theo}\label{theo:MainTheorem}
	Let $G$ be a finite group admitting a holomorphic and translation-free action on a three-dimensional abelian variety  $A=\mathbb C^3/\Lambda$ such that the fixed locus is non-empty and finite and  the volume form of $A$ is preserved. Then, $G$ is cyclic of order $3$ or $7$, or one of the groups
	\[\He(3)=\langle g,h,k\:\mid\: g^3=h^3=k^3=[g,k]=[h,k]=1,\:[g,h]=k\rangle,\quad \ZZ_3^2=\langle h, k\rangle<\He(3).\]
	There are exactly eight biholomorphism classes of  quotients $A/G$. They are pairwise topologically  distinct. The table below contains precisely one representative $Z_i$ for each class.
	\medskip
	\begin{center}
		{\footnotesize
		\bgroup\def\arraystretch{1.5}\begin{tabular}{|c|c|c|l|c|c|} \hline 
		$i$	&\emph{$G$} & \emph{ $\Lambda$} & \multicolumn{1}{c|}{\emph{action}} & \emph{singularities} & $\pi_1(Z_i)$ \\ \hline \hline
		1	&$\ZZ_7$ & $\Lam(\ze_7,\ze_7^2,\ze_7^4) $ & \begin{tabular}{l}$\Phi(1)(z)=\diag(\ze_7,\ze_7^2,\ze_7^4)\cdot z$\end{tabular}& $7\times \tfrac{1}{7}(1,2,4)$ & $\{1\}$ \\ \hline \hline
		2	&$\ZZ_3$ & $\ZZ[\ze_3]^3$ & \begin{tabular}{l}$\Phi(1)(z)=\diag
		(\ze_3,\ze_3,\ze_3)\cdot z $\end{tabular} & $27\times \tfrac{1}{3}(1,1,1)$ & $\{1\}$\\ \hline \hline
		3	&$\ZZ_3^2$ & $\ZZ[\ze_3]^3$ & \begin{tabular}{l}$\Phi(h)(z)=\diag(1,\ze_3^2,\ze_3)\cdot z+(t,t,t)$\\  $\Phi(k)(z)=\diag(\ze_3,\ze_3,\ze_3)\cdot z$\end{tabular}& $9\times \tfrac{1}{3}(1,1,1)$ & $\ZZ_3$\\ \hline
		4   & $\ZZ_3^2$ & $\ZZ[\ze_3]^3+\ZZ(t,t,0) $ & \begin{tabular}{l}$\Phi(h)(z)=\diag(1,\ze_3^2,\ze_3)\cdot z+\tfrac{1}{3}(1,1,3t)$\\  $\Phi(k)(z)=\diag(\ze_3,\ze_3,\ze_3)\cdot z$\end{tabular} &  $9\times \tfrac{1}{3}(1,1,1)$ & $\ZZ_3$\\ \hline
		5   & $\ZZ_3^2$ & $\ZZ[\ze_3]^3+\ZZ(t,t,t) $ & \begin{tabular}{l}$\Phi(h)(z)=\diag(1,\ze_3^2,\ze_3)\cdot z+\tfrac{1}{3}(1,1,1)$\\  $\Phi(k)(z)=\diag(\ze_3,\ze_3,\ze_3)\cdot z$\end{tabular} &  $9\times \tfrac{1}{3}(1,1,1)$ & $\ZZ_3$ \\ \hline
		6   & $\ZZ_3^2$ & $\ZZ[\ze_3]^3+\ZZ(t,t,t)+\ZZ(t,-t,0) $ & \begin{tabular}{l}$\Phi(h)(z)=\diag(1,\ze_3^2,\ze_3)\cdot z+\tfrac{1}{3}(1,1,2)$\\  $\Phi(k)(z)=\diag(\ze_3,\ze_3,\ze_3)\cdot z$\end{tabular} &  $9\times \tfrac{1}{3}(1,1,1)$ & $\ZZ_3$\\
		 \hline \hline
		7 &$\He(3)$ & $\ZZ[\ze_3]^3+\ZZ(t,t,t)$ & \begin{tabular}{l} $\Phi(g)(z)=\begin{pmatrix}
				0 & 0 & 1\\ 1 & 0 & 0 \\ 0 & 1 & 0
			\end{pmatrix} \cdot z+ (t,0,0) $ \\ $\Phi(h)(z)=\diag(1,\ze_3^2,\ze_3)\cdot z+ \tfrac{2}{3}(1,1,1)$ \end{tabular} & $3\times \tfrac{1}{3}(1,1,1)$ & $\ZZ_3^2$\\ \hline
		8 & $\He(3)$ & $\ZZ[\ze_3]^3+\ZZ(t,t,t)+\ZZ(t,-t,0)$ & \begin{tabular}{l} $\Phi(g)(z)=\begin{pmatrix}
				0 & 0 & 1\\ 1 & 0 & 0 \\ 0 & 1 & 0
			\end{pmatrix}\cdot z + (t,0,0) $ \\ $\Phi(h)(z)=\diag(1,\ze_3^2,\ze_3)\cdot z+ \tfrac{2}{3}(1,1,1)$ \end{tabular} & $3\times \tfrac{1}{3}(1,1,1)$ & $\ZZ_3^2$\\ \hline
		\end{tabular}\egroup}
	\end{center}	
	\medskip
	In the table, $t:=(1+2\ze_3)/3$ and $\Lam(\ze_7,\ze_7^2,\ze_7^4)$ denotes  the lattice with basis $\{(\ze_7^k,\ze_7^{2k},\ze_7^{4k})\mid k=0,\ldots,5\}.$
	All quotients are Gorenstein threefolds with trivial canonical class and admit smooth rigid Calabi-Yau threefolds as crepant resolutions.
The quotients $Z_1$ and $Z_2$ are simply connected and  $Z_3, \ldots, Z_8$  are uniformized by $Z_2$.
\end{theo}

As we pointed out above, all abelian varieties in the theorem, with the exception of the first, are isomorphic to $E^3$. 
However, if we want to represent  the varieties $Z_4,\ldots, Z_8$ as quotients of $E^3$, different choices of the linear parts of the action would be required due to  necessary  changes of bases.
The quotients $Z_3$ and $Z_5$ have already been constructed and distinguished by Bauer and the first author in \cite{BG21}.  This paper was, 
besides  the mentioned results of Oguiso et al., our primary motivation to derive  the complete fine classification. \\
We will now explain how the paper is organized: in Section~\ref{sec:preliminaries}, we introduce some notations and definitions regarding holomorphic group actions on abelian threefolds. We also 
summarize the known  results of Oguiso et al., in particular the classification of all possible Galois groups. 
 In  Section \ref{sec:Kristallweizen}, we introduce certain  methods and techniques from crystallographic group theory, such as  Bieberbach's theorems.  
 They are in  fact our main tools to establish 
 a fine classification of Calabi-Yau threefolds of type $\mathrm{III}_0$, i.e., to finish the  proof of Theorem 
 \ref{theo:MainTheorem} for the quotients modulo the groups 
$\mathbb Z_3^2$ and  $\He(3)$. The reason why we can apply these  methods relies  on the 
observation that the fundamental group of the smooth locus of quotient of an abelian variety by a finite group $G$ acting freely in codimension one
is  crystallographic. Similarly to the case of \'etale torus quotients, i.e. flat K\"ahler manifolds,  Bieberbach's  results  allow us to decide if two such quotients are biholomorphic or homeomorphic, respectively
(cf.~\cite{HalendaLutowski}).
In Section~\ref{sec:FineClass}, we adapt the \qq{classification machinery} outlined in \cite{DG} to the singular case and finally reach our classification result. A holomorphic action of a group $G$ on an abelian variety  $A=\CC^3/\Lam$ is given by affine transformations where the linear parts, which form the so-called \textit{analytic representation} $\rho$, leave the lattice $\Lambda$ invariant. Hence, $\Lam$ can be viewed as $G$-module via this action. We see in Section~\ref{sec:preliminaries} that, after fixing a basis, the analytic representation $\rho$ is uniquely determined for both groups $\ZZ_3^2$ and $\He(3)$. For specifying the possible abelian varieties, we therefore search for those lattices admitting a $G$-module structure via $\rho$. Then, we determine all actions on these abelian varieties, i.e., the translation parts of the actions. In contrast to the linear part, the translation part is not a homomorphism, but a 1-cocycle. 
Thus, it defines a cohomology class in the group cohomology $H^1(G,A)$ with additional properties reflecting the freeness of the action in codimension two.
 Finally, analyzing the orbits of these cohomology classes under certain group actions enables us to distinguish the biholomorphism and homeomorphism classes of the quotients. The calculations will be performed computer aided using the computer algebra system MAGMA (\cite{MAGMA}). Our code can be found on the website:
\begin{center}
\url{http://www.staff.uni-bayreuth.de/~bt300503/publi.html}.
\end{center}

\bigskip
\textbf{Notation.} % We use the standard notation from complex geometry and representation theory of finite groups. 
The cyclic group of order $n$ is denoted by $\ZZ_n$ and
the group of affine linear transformations of $\mathbb K^n$ by $\AGL(n,\mathbb K)$.
 The group of biholomorphic transformations of a complex torus or abelian variety $T$ is $\Bihol(T)$, whereas $\Aut(T)$ is the subgroup of $\CC$-linear biholomorphisms and  
$\Aut_{\sC^\infty}(T)$ the group of $\mathbb R$-linear diffeomorphisms. By the set of \textit{fixed points} or the \textit{fixed locus} of a $G$-action on a torus $T$, we mean the set of all elements in $T$ having non-trivial stabilizer group.

% *************************************

\section{Preliminaries and Background}\label{sec:preliminaries}

In this short section,
we recall the classification of finite groups $G$ which admit a
holomorphic action on a  three-dimensional abelian variety $A$ such that 
\begin{itemize}
\item[(a)]
the fixed locus is non-empty and isolated, 
\item[(b)] 
 the action preserves the volume-form of $A$.
 \end{itemize}
 
 Since holomorphic maps between complex tori, so in particular between abelian varieties, are affine, we can decompose the action of an element $g\in G$ into its \textit{linear part} $\rho(g)$ and its \textit{translation part} $\tau(g)$, i.e., $\Phi(g)(z) = \rho(g)z + \tau(g)$.  By viewing the linear parts of the action as automorphisms of $\CC^{3}$, we obtain a representation
\begin{align*}
	\rho \colon G \to \GL(3,\CC),
\end{align*}
which is called the \textit{analytic representation}. As the quotient of an abelian variety by a finite group of translations is again an abelian variety, we can and will always assume that $G$ acts without translations, equivalently $\rho$ is faithful.\\
Condition~$(\rm{b})$ amounts to saying that the
 analytic representation $\rho$ maps to the special linear group $\SL(3,\CC)$, hence, $X=A/G$ is Gorenstein with trivial  canonical divisor. \\
 
 We distinguish between simply connected quotients $X=A/G$ and such with non-trivial fundamental group. 
 In the first case, not only the groups, but even the isomorphism types of the quotients are completely classified by Roan, Yau and Oguiso, see \cite{RoanYau},  \cite{RoanKummer}, \cite{OguisoSimplyConnected}:

\begin{theo}\label{theo:SimplyConnected}
	Suppose that $X=A/G$  is simply connected, then, $G$ is isomorphic to $\ZZ_3$ or $\mathbb Z_7$. The quotient $X$ is biholomorphic to 
	\[
	Z_1=\Jac(Q)/\langle \diag(\ze_7,\ze_7^2,\ze_7^4)\rangle  \qquad \makebox{or} \qquad  Z_2= E^3/\langle \ze_3\cdot \id\rangle . 
		\]
\end{theo}

\begin{rem}\label{rem:RY}
The quotient $Z_2$ has $27$ singularities of type $\tfrac{1}{3}(1,1,1)$ given by the images of the $27$ fixed points of the automorphism $\ze_3\cdot\id$ on $E^3$ under the quotient map. Each singularity admits a unique toric crepant resolution with exceptional divisor $\PP^2$. Similarly, $Z_1$ has seven singularities of type $\tfrac{1}{7}(1,2,4)$. They also have a unique crepant resolution which is the composition of three toric blow-ups; the exceptional locus consists of three Hirzebruch surfaces $\mathcal H_2$.
These procedures lead to global crepant resolutions $\psi_i\colon\hat{Z}_i\to Z_i$. The smooth threefolds  $\hat{Z}_i$ are   Calabi-Yau: indeed, the canonical divisor of $\hat{Z}_i$ is trivial since $\psi_i$ is crepant and $K_{Z_i}$ is trivial by construction. Moreover, the irregularity  $q_1(\hat{Z}_i)$ is zero because there are no invariant holomorphic one-forms on $E^3$ and $\Jac(Q)$, respectively. Furthermore, $\hat{Z}_i$ is rigid since $h^{1,2}(\hat{Z}_i)=0$. For details, see \cite{RoanYau}.
\end{rem}

\begin{rem}
	Recently, Gachet \cite{Gachet} generalized Theorem~\ref{theo:SimplyConnected} to higher dimensions:
	let $A$ be an abelian variety of dimension $n\geq 3$ and $G$ a finite group acting freely in codimension two on $A$. If the quotient $A/G$ has a simply connected Calabi-Yau manifold as crepant resolution, then $A$ is isogenous to $E^n$ or $E_{u_7}^n$. Here, $E$ is the Fermat elliptic curve and $E_{u_7}$ the elliptic curve corresponding to the lattice generated by $1$ and $u_7:=\ze_7+\ze_7^2+\ze_7^4$.  
	The group $G$ is generated by its elements having fixed points on $A$.	
Note that $E_{u_7}^3\simeq \Jac(Q)$.
\end{rem}

While Theorem~\ref{theo:SimplyConnected} settles the classification in the simply connected case, only the isomorphism classes of the  groups and the analytic representations are known 
if $X=A/G$ is not simply connected. In fact, Oguiso and Sakurai proved in \cite[cf. Theorem~3.4]{OguisoQuotientType}: 

\begin{theo}\label{theo:NotSimply}
 If $X=A/G$ is not simply connected, then $G$ is isomorphic to $\ZZ_3^2$ or $$\He(3)=\langle g,h,k\:\mid\: g^3=h^3=k^3=[g,k]=[h,k]=1,\:[g,h]=k\rangle.$$ More precisely, the following holds:
	\begin{enumerate}
		\item If $G=\ZZ_3^2=\langle h,k\rangle$, then the analytic representation is equivalent to 
		\[\rho(h)=\diag(1,\ \ze_3^2,\ \ze_3),\quad\rho(k)=\diag(\ze_3,\ \ze_3,\ \ze_3).\]
		\item If $G=\He(3)$, then the analytic representation is equivalent to
		\[
		\rho(g)=\begin{pmatrix} 0&0&1\\1&0&0\\0&1&0 \end{pmatrix},\quad \rho(h)=\begin{pmatrix} 1&&\\&\ze_3^2&\\&&\ze_3 \end{pmatrix},\quad \rho(k)=\begin{pmatrix} \ze_3&&\\&\ze_3& \\ &&\ze_3 \end{pmatrix}.
		\]
	\end{enumerate}
\end{theo}

\begin{rem}\label{rem:resolutions}
	Since we assume that the action is free in codimension two, the unique non-trivial elements having fixed points are $k$ and $k^2$.
	Indeed, the linear parts of all other non-trivial elements have eigenvalue one, hence, the fixed locus of such an element would have positive dimension if it was non-empty. In particular, the map $Z_2\simeq A/\langle k\rangle\to A/G$ is an unramified cover, i.e., $X=A/G$ is uniformized by $Z_2$ and has fundamental group $\pi_1(X)\simeq G/\langle k\rangle$.\\
	Similar to Remark~\ref{rem:RY}, all quotients $X=A/G$ admit unique crepant resolutions 
	$\psi \colon \hat{X} \to X$ leading to Calabi-Yau threefolds $\hat{X}$. These resolutions fit in commutative diagrams:
	\[
		\begin{tikzcd}
		 	\hat{Z}_2 \arrow{d}\arrow{r} & \hat{X} \arrow{d}\\
		 	Z_2 \arrow{r}& X.
		\end{tikzcd}
	\]
	As $\hat{Z}_2\to\hat{X}$ is also an unramified cover, the manifold $\hat{X}$ is rigid as well.
\end{rem}	

We point out that there are no other varieties uniformized by $Z_2$:

\begin{prop}\label{prop:QuotientsBeauville}
	Let $X$ be a complex variety uniformized by $Z_2=E^3/\langle \ze_3\cdot\id\rangle$. Then, $X$ is Gorenstein with trivial canonical class and biholomorphic to a quotient of a complex three-dimensional abelian variety by a faithful action of a group $G$ which is $\ZZ_3$, $\ZZ_3^2$ or $\He(3)$.
\end{prop}

\begin{proof}
	Let $q\colon Z_2\to X$ be the universal cover and $pr\colon E^3\to Z_2=E^3/\ZZ_3$  the quotient map.  Then, the map $q\circ pr\colon E^3\to X$ is Galois  because 	
 every $f\in \Deck(q)$ can be lifted to $E^3$ as we will see later on in Proposition~\ref{prop:ConsBieb}.  In other words,  $X$ is a quotient of $E^3$.\\
 	Since $q$ is unramified and $\chi(\sO_{Z_2})=0$, we get $\chi(\sO_X)=0$. As $h^i(\sO_X)\leq h^i(\sO_{Z_2})=0$ for $i=1, 2$, we conclude $0=\chi(\sO_X)=1-p_g(X)$. Therefore, $X$ has trivial canonical class because the volume form of $E^3$ must be invariant under the action of $\Deck( q\circ pr)$.\\
Let $H\trianglelefteq \Deck( q\circ pr)$ be the subgroup of translations, $G:=\Deck( q\circ pr)/H$  and $A=E^3/H$. Then, $G$ acts translation-free on $A$ and $X\simeq A/G$. It follows from Theorem~\ref{theo:SimplyConnected} and Theorem \ref{theo:NotSimply} that the group $G$ is isomorphic to  $\ZZ_3$, $\ZZ_3^2$ or $\He(3)$.
\end{proof}

% ************************************************************

\section{Crystallographic Groups and Torus Quotients}\label{sec:Kristallweizen}

In this section, we collect the tools form crystallographic group theory that we need to derive the fine classification. Since they also work for non-algebraic complex tori and in arbitrary dimension, we discuss them in the general setting of a complex torus $T=\CC^n/\Lam$ of dimension $n\geq 2$.
In particular, we explain how to use Bieberbach's structure theorems to decide if two torus quotients, which we view in the non-algebraic case as normal complex spaces in the sense of Cartan (cf. \cite{Cartan}), are biholomorphic or homeomorphic.
Our references for  the theory of crystallographic groups are the textbooks \cite{charlap} and \cite{Szcz}.
%If a finite group acts freely on a complex torus, the fundamental group of the quotient is a Bieberbach group and is useful for determining the biholomorphism and homeomorphism type. In our situation, the action is not free. Here, we work instead of the fundamental group with the so-called \textit{orbifold fundamental} group, which we introduce now. We then collect some important properties of this group and see in particular that it is crystallographic.

\begin{definition}\label{sec:bieberbach}
	Let $T=\CC^n/\Lambda$ be a complex torus and $G$ a finite group of biholomorphisms acting on $T$ without translations. Let $\pi \colon \CC^n \to T$ be the universal cover, then, we define 
	the \emph{orbifold fundamental group} as
	\[
	\pi^{orb}_1(T,G):=\lbrace \gamma \colon \CC^n \to \CC^n \mid \exists\, g \in G ~s.t. ~ \pi \circ \gamma =g\circ \pi \rbrace. 
	\]
\end{definition}

\begin{rem}\
\begin{enumerate}
	 \item If $G$ acts freely in codimension at least one, then $\Sing(X)=p(F)$, where $p\colon T \to T/G=X$ is the quotient map and 
	 $ F:=\lbrace x\in T \mid \exists\, g \in G\setminus \{1\} ~ s.t. ~ g(x)=x \rbrace$ is the set of points with non-trivial stabilizer group. 
	 By assumption, the analytic set $F$ has codimension at least two.
	 Since the restriction  $p\colon T\setminus F \to X\setminus \Sing(X)$ is finite and unramified,  the composition 
	 \[
	 \CC^n\setminus \pi^{-1}(F) \to T\setminus F \to X\setminus \Sing(X)
	 \]	
	 is  an unramified cover.  It is universal because $\CC^n\setminus \pi^{-1}(F)$ is simply connected (cf. \cite[p.~378]{Prill}). The Galois group of this cover, i.e., the fundamental group of  $X^\circ=X\setminus \Sing(X)$, equals the orbifold fundamental group $\Gamma:=\pi^{orb}_1(T,G)$.	 
	\item By construction, $\Gamma$ is a group of affine transformations. As $G$ acts translation-free on $T$, the lattice $\Lam$ of the torus  coincides with the subgroup of translations of $\Gamma$ and $G\simeq \Gamma/\Lam$. Moreover, $\Gamma$ is discrete and cocompact, and $X=T/G \simeq \CC^n/\Gamma$ as complex varieties or spaces.
	\item Since $G$ is finite, we may assume that 
	the analytic representation $\rho$ is unitary. 
	Via the identification
	\[
	\CC^n\to\RR^{2n},\quad (z_1,\ldots,z_n)\mapsto (x_1,y_1,\ldots,x_n,y_n),\qquad\mathrm{where}\quad z_j=x_j+\sqrt{-1}y_j,
	\]
	we can  view $\rho$ as a real representation $\rho_{\RR} \colon G \to O(2n)$ and consider $\Gamma=\pi^{orb}_1(T,G)$ as a subgroup of the Euclidean group $\mathbb E(2n) =\RR^{2n} \rtimes O(2n)$.
\end{enumerate}
 \end{rem}

\begin{definition}
	A discrete cocompact subgroup of the Euclidean group $\mathbb E(n)$ is called a \emph{crystallographic group}.
\end{definition}

The above remark tells us that  the orbifold fundamental group $\pi^{orb}_1(T,G)$ is crystallographic if the action of $G$ is free in codimension one. \\
Next, we recall the structure theorems of Bieberbach, which will be crucial in the sequel.

\begin{theo}[\cite{Bieber1, Bieber2}]\label{theo:Bieb}
	The  translation subgroup $\Lambda:=\Gamma \cap \mathbb R^n$ of a crystallographic group  $\Gamma \leq \mathbb{E}(n)$ is a 
	lattice  of rank $n$ and the quotient $\Gamma/\Lambda$ is finite.  All normal Abelian subgroups of $\Gamma$ are contained in $\Lambda$. Furthermore, an isomorphism between two crystallographic groups is given by conjugation with an affine transformation.
\end{theo}

The next lemma  gives a link between homeomorphisms of torus quotients and their orbifold fundamental groups. This will allow us to apply Bieberbach's structure theorems to distinguish the quotients $T/G$ as we will explain in the proposition below. 

\begin{lemma}\label{le:homeoRestricts}
	Every homeomorphism between two torus quotients by actions that are free in codimension  one restricts to a homeomorphism of their smooth loci. In particular, it induces an isomorphism of the orbifold fundamental groups.
\end{lemma}

\begin{proof}
	Let $X=T/G$ and $X'=T'/G'$ be torus quotients  and $f\colon X\to X'$ a homeomorphism. Then, $f$ induces an isomorphism of the local fundamental groups $\pi_1^{loc}(X,p)$ and $\pi_1^{loc}(X',f(p))$ for every $p\in X$. Since the actions are free in codimension one, these groups are trivial if and only if the points $p$ and $f(p)$, respectively, are smooth. Therefore, $f$ maps the smooth locus of $X$ homeomorphically to the smooth locus of $X'$.
\end{proof}

%Using this lemma, the theorems of Bieberbach have immediate geometric consequences that help us to distinguish the quotients $T/G$.

\begin{prop}\label{prop:ConsBieb}
	Let $\Phi\colon G\to \Bihol(T)$ and $\Phi'\colon G'\to\Bihol(T')$ be translation-free holomorphic actions of finite groups $G$ and $G'$ which are free in codimension one. Assume that the quotients $X=T/G$ and $X'=T'/G'$ are homeomorphic. Then:
	\begin{enumerate}
		\item The groups $G$ and $G'$ are isomorphic.
		\item There exists an affine transformation $\alpha\in \AGL(2n,\RR)$ inducing homeomorphisms $\widehat{\alpha}$ and $\widetilde{\alpha}$, such that the following diagram commutes:
	 \[
	 \begin{tikzcd}
	 	T \arrow{d}\arrow{r}{\widetilde{\alpha}} & T' \arrow{d}\\
	 	X\arrow{r}{\widehat{\alpha}} & X'.
	 \end{tikzcd}
	\]
 	\end{enumerate}
	Furthermore, any biholomorphism $f\colon X\to X'$ lifts to a biholomorphism of the tori, i.e., it is induced by an affine transformation $\alpha\in\AGL(n,\CC)$.
\end{prop}

\begin{proof}
	The assertions under the assumption that the quotients are homeomorphic follow directly from Theorem~\ref{theo:Bieb} using Lemma~\ref{le:homeoRestricts}. So, let $f\colon X\to X'$ be a biholomorphic map. Clearly, this map restricts to a biholomorphism  between the smooth loci of the quotients, $f\colon X^\circ\to (X')^\circ$, and lifts to the universal covers:
		\begin{equation*}
		\begin{tikzcd}
			\CC^n\setminus \pi^{-1}(F) \arrow{d}\arrow{r}{\widetilde{f}} & \CC^n\setminus (\pi')^{-1}(F')\arrow{d}\\
			X^\circ \arrow{r}{f} & (X')^\circ.
		\end{tikzcd}	
	\end{equation*}
	Since $\pi^{-1}(F)\subset \CC^n$ is analytic and of codimension at least two, there exists a unique biholomorphic extension $\widetilde{F}\colon \CC^n\to\CC^n$ of $\widetilde{f}$ by Riemann's second extension theorem (cf. \cite[Theorem~6.12]{Grauert}). Because of the commutativity of the diagram, we can find for any $\gamma\in\Gamma$ an element $\gamma'\in\Gamma'$ such that $\widetilde{F}\circ \gamma=\gamma'\circ\widetilde{F}$ for all $z\in \CC^n\setminus\pi^{-1}(F)$, hence, for all $z\in\CC^n$ by the identity theorem.
	Therefore, the biholomorphism $\widetilde{F}$ induces a biholomorphic map $F\colon X\to X'$, that coincides with $f$.  Since conjugation with $\widetilde{F}$ gives an isomorphism between the crystallographic groups $\Gamma$ and $\Gamma'$, the biholomorphism $\widetilde{F}$ maps the lattice $\Lam$ to $\Lam'$ and induces a well-defined map between the tori $T$ and $T'$. In particular, $\alpha:=\widetilde{F}$ is affine linear.
\end{proof}

As in \cite[Remark~4.7]{DG}, we make use of the following observations and notations.

\begin{rem}\label{rem:AffinitiesBieb}
	Let $f\colon X\to X'$ be a homeomorphism induced by an affine transformation $\alpha(x)=Cx+d$. Then, the commutativity of the diagram in Proposition~\ref{prop:ConsBieb} is equivalent to the existence of an isomorphism $\varphi\colon G\to G'$ such that
	\[ (\mathrm{a})\: C\rho_\RR(u)C^{-1}=\rho'_\RR(\varphi(u)) \qquad \mathrm{and}\qquad (\mathrm{b})\: (\rho'_\RR(u)-\id)d=C\tau(\varphi^{-1}(u))-\tau'(u) \]
	hold for all $u\in G$, where the second item is an equation holding on $T'$.\\
	Item $($a$)$ means that $\rho_\RR$ and $\rho'_\RR\circ \varphi$ are equivalent as real representations or as complex representations if $f$ is holomorphic. Together with $C\Lam =\Lam'$, this precisely means that $C$ is contained in
	\[ \sN_\RR(\Lam,\Lam'):=\{C\in\GL(2n,\RR)\mid C \Lam=\Lam',\: C\cdot\im(\rho_\RR)=\im(\rho'_\RR)\cdot C\}\]
	in the homeomorphic case, and in 
	\[\sN_\CC(\Lam,\Lam'):=\sN_\RR(\Lam,\Lam')\cap \GL(n,\CC)\]
	in the holomorphic case. Note that $\varphi=\varphi_C$ is uniquely determined by $C$ due to the faithfulness of the analytic representations.
\end{rem}	

Next, we want to explain the meaning of item $($b$)$ of the previous remark in more detail. For this, we first analyze some properties of the translation part $\tau\colon G\to T$ of an action of $G$ on $T$.

\begin{rem}\
\begin{enumerate}
\item The translation part $\tau\colon G\to T$ of a holomorphic action $\Phi(g)(z)=\rho(g)z+\tau(g)$ on a complex torus $T$ is not a homomorphism, but a 1-cocycle,
\[\tau(gh)=\rho(g)\tau(h)+\tau(g).\]
Thus, it defines a class in the first group cohomology,
\[H^1(G,T)=\frac{\{\tau\mid \tau(gh)=\rho(g)\tau(h)+\tau(g)\}}{\{\tau\mid \exists\, d\in T\colon \tau(g)=\rho(g)d-d\}}, \]
where we view $T$ as a $G$-module via the action of $\rho$. Up to conjugation by a translation, $\Phi$ is uniquely determined by $\rho$ and the cohomology class of $\tau$.
\item For any $C\in \sN_\RR(\Lam,\Lam')$, the expression $C\ast \tau:=C\cdot (\tau\circ\varphi^{-1})$  is a cocycle on $T'$. The equation in Remark~\ref{rem:AffinitiesBieb}(b) means that $C\ast\tau$  belongs to the same cohomology class as $\tau'$.
\item	 In the special case where $\rho=\rho'$  and  $T=T'$, the sets $\mathcal N_{\mathbb R}(\Lambda,\Lambda)$ and 
	$\mathcal N_{\mathbb C}(\Lambda,\Lambda)$ are the normalizers of $\im(\rho_{\mathbb R})$ in  $\Aut_{\sC^\infty}(T)$ and in  $\Aut(T)$. 
	For simplicity, we denote them by $\mathcal N_{\mathbb R}(\Lambda)$ and $\mathcal N_{\mathbb C}(\Lambda)$. 
	Item (b) tells us that they  act on $H^1(G,T)$ by $C \ast \tau$.
	It follows that $X$ and $X'$ are homeomorphic (or biholomorphic) if and only if the cohomology classes of $\tau$ and $\tau'$ belong to the same orbit under this action.
\end{enumerate}
\end{rem}

% **********************************************************************

\section{Fine Classification of the Quotients}\label{sec:FineClass}

This section is devoted to the fine classification of the quotients by the groups $\ZZ_3^2$ and $\He(3)$. The methodology is inspired by \cite{DG}.  
Note that the linear part of an action of a finite group $G$ on an abelian variety $A=\CC^3/\Lam$ leaves the lattice $\Lam$ invariant. Hence, $\Lam$ becomes a $G$-module via the action of the analytic representation. Using this observation and the results from the previous section,  a classification can be achieved performing the following steps:\\

\noindent 
{\bf Scheme for the Classification}

\begin{enumerate}
	\item \label{scheme-first}
	For each group $G=\ZZ_3^2$ and $\He(3)$, determine all lattices $\Lambda$ which have a  $G$-module structure via the 
	representations $\rho$ from  Theorem~\ref{theo:NotSimply}.
	\item 
	For each $A=\mathbb C^3/\Lambda$ and for  all  cohomology classes in  $H^1(G,A)$ leading to an action with non-empty and finite fixed locus, fix a  
	 representative $\tau$.  We will refer to such classes as \textit{good} cohomology classes.
	\item
	Decide, which quotients of the abelian varieties $\CC^3/\Lam$ by the actions $\Phi(u)(z)=\rho(u)z+\tau(u)$ are biholomorphic or homeomorphic, respectively. 
\end{enumerate}

\begin{rem}
Note that the abelian varieties $A$ defined by the lattices $\Lambda$ from the first step are all isomorphic to $E^3$. However, $\Lambda$ and $\ZZ[\ze_3]^3$ are in general not isomorphic as 
$G$-modules with module structure defined by $\rho$. Hence, we can not assume that $A$ is equal to $E^3$.
\end{rem}

We follow our classification strategy and first determine the possible lattices $\Lambda$. For this, we show that each candidate for $\Lambda$ contains $\mathbb Z[\zeta_3]^3$ as a sublattice of finite index. 

\begin{prop} \label{prop:isog}
	Let $A$ be a three-dimensional abelian variety allowing an action of one of the groups $\ZZ_3^2$ or $\He(3)$ with analytic representation $\rho$. Then, the abelian variety $A$ is equivariantly isogenous to a product of three elliptic curves $E_i \subset A$, each of which is a copy of the Fermat elliptic curve $E=\CC/\mathbb Z[\zeta_3]$. 
\end{prop}

\begin{proof}
	Consider the following subtori 
	\[E_1 := \ker(\rho(h)-\id_A)^0, \quad E_2 := \ker(\rho(hk)-\id_A)^0, \quad E_3 := \ker(\rho(hk^2)-\id_A)^0, \]
	where the superscript $0$ denotes the connected component of the identity.
	By construction, the curves $E_i$ are isomorphic to $E=\CC/\mathbb Z[\zeta_3]$ and the addition map 
	\[
	\mu \colon E_1 \times E_2 \times E_3 \to A
	\]
	is an equivariant isogeny.
\end{proof}

Due to Proposition~\ref{prop:isog}, we can assume that  $A = E^3/K$, where $E = \mathbb C/\mathbb Z[\zeta_3]$ is the Fermat elliptic curve and $K$ is the 
kernel of the addition map $\mu$. 
Sometimes, if we need to keep track of the order, we may write $E_i$ for the $i$-th factor of $E^3$. 
Note that 
	\[E_i\hookrightarrow E_1\times E_2\times E_3 \rightarrow \bigslant{E^3}{K}=A\]
	is injective for all $i=1,2,3$. Hence, the kernel $K$ does not contain non-zero multiples of unit vectors. 

\begin{rem}\label{rem:actionStandard}
	Let $\Phi \colon G \hookrightarrow \Bihol(A)$ be a faithful holomorphic action. 
	\begin{itemize}
	\item
	If $G = \Heis(3)$, then, up to a change of the origin in $A$, the translation part $\tau \colon G \to A$ of $\Phi$ can be written in the form
	\[
		\tau(h)= \left( a_1, \ a_2, \ a_3\right), \quad
		\tau(k)= \left( 0, \ 0, \ 0\right), \quad 
		\tau(g)= \left( b_1, \ b_2, \ b_3\right).
	\]
	\item 
	If $G=\mathbb Z_3^2=\langle h,k\rangle$, then the translation part of $\Phi$ can be written as 
	\[
		\tau(h)= \left( a_1, \ a_2, \ a_3\right), \quad
		\tau(k)= \left( 0, \ 0, \ 0\right).
	\]
\end{itemize}	
We refer to such a translation part as a
\emph{cocycle in standard form}.  
%Sometimes, we may write  $u(z)$ instead of $\Phi(u)(z)$ for $u \in G$, by a slight abuse of notation. 
\end{rem}

\begin{lemma}\label{le:a_i}
	Let $\tau$ be a cocycle in standard form, then, $a_1,a_2$ and $a_3$ belong to $E[3]$, the set of $3$-torsion points of $E$.
\end{lemma}

\begin{proof}
	Let $\Phi$ be the action corresponding to $\tau$. Since $\Phi$ is a homomorphism and $h$ is an element of order three, we have $\Phi(h)^3=\id$.  This property is assured if and only if $\tau(h^3)=(3a_1,\ 0,\ 0)$ equals zero in $A$, which is equivalent to require $3a_1=0$ in $E$. Analogously, $\tau((hk)^3)=(0,\ 3a_2,\ 0)$ and $\tau((hk^2)^3)=(0,\ 0,\ 3a_3)$ are zero in $A$, which forces $3a_2=3a_3=0$ in $E$.
\end{proof}

\begin{lemma} \label{le:well-defined}
	Let $\tau$ be a  cocycle in standard form. Then:
	\begin{enumerate}
		\item If $G = \mathbb Z_3^2$, the following element is zero in $A$:
		\begin{itemize}
			\item[]  $v_1 := \left((\ze_3-1)a_1, \ (\ze_3-1)a_2, \ (\ze_3-1)a_3\right)$.
		\end{itemize}
		Conversely, given $a_i\in E[3]$ such that $v_1$  is zero in $A$, we obtain a 
		cocycle $\tau\colon \mathbb Z_3^2\to A$ in standard form.
		\item If $G = \Heis(3)$, the elements $v_1, ..., v_4$ are zero in $A$, where
		\begin{itemize}
			\item[] $v_2 := \left(b_1+b_2+b_3, \ b_1+b_2+b_3, \ b_1+b_2+b_3\right)$, 
			\item[] $v_3 := \left((\ze_3-1)b_1, \ (\ze_3-1)b_2, \ (\ze_3-1)b_3\right)$, 
			\item[] $v_4 := \left(\ze_3a_1-a_3+(\ze_3-1)b_1, \ \ze_3a_2-a_1, \ \ze_3a_3-a_2+(\ze_3^2-1)b_3\right)$
		\end{itemize}
		and $v_1$ is as above.\\
		Conversely, given $a_i\in E[3]$ and $b_j$ such that $v_1, \ldots , v_4$  are zero in $A$, we obtain a 
		cocycle $\tau\colon \He(3)\to A$ in standard form.
	\end{enumerate}
\end{lemma}

\begin{proof}
	We will only sketch the proof of (2). For this, denote by $\Phi$ the corresponding action. Since $\Phi$ is a homomorphism, the images $\Phi(g)$, $\Phi(h)$ and $\Phi(k)$ of the generators fulfill the six defining relations of $\Heis(3)$.  As an example, the relation $\Phi(g) \circ \Phi(k) = \Phi(k) \circ \Phi(g)$ precisely means that the difference $\tau(gk)-\tau(kg)=v_3$ is zero in $A$. Since $\Phi(k)^3=\id_A$ is always true and $\Phi(h)^3=\id_A$ as  $a_1$ belongs to $E[3]$, we only get four conditions.
\end{proof}

\begin{cor}\label{cor:p(E)K}
	Let $G$ be one of the groups $\ZZ_3^2$ or $\He(3)$ and let $\tau$ be a good cocycle in standard form. Denote by $p_i\colon K\to E_i$ the projection on the $i$-th factor. Then, the set $E[3]$ is not contained in $p_i(K)$ for all $i=1,2,3$.
\end{cor}

\begin{proof}
	If we assume that $E[3]\subseteq p_i(K)$, then we can find an element in $K$ whose $i$-th coordinate equals $a_i$. Therefore, the elements $h$ (for $i=1$), $hk$ (for $i=2$), or $hk^2$ (for $i=3$) have fixed points, respectively that are not isolated since the linear parts of the actions of these elements have the eigenvalue 1. This contradicts  the assumption of $\tau$ to be good.
\end{proof}

\begin{prop}\label{prop:Kfixed}
	Each coordinate of any element in $K$ is fixed by multiplication with $\ze_3$, i.e., is contained in the group
\[\Fix_{\ze_3}(E):=\{z\in E\mid \ze_3\cdot z =z\}=\{0,t,-t\},\quad\mathrm{where}\quad t:=\tfrac{1}{3}(1+2\ze_3).\]
\end{prop}

\begin{proof}
	Let $G$ be one of the groups $\ZZ_3^2$ or $\He(3)$. For $u\in G$, we may view $\rho(u)$ as an automorphism of $E^3$ mapping $K$ to $K$. Let $(t_1,t_2,t_3)\in K$. Then, each coordinate $t_i$ is a $3$-torsion point of $E_i=E$: this follows from
	\begin{align*}
	(\rho(hk)-\id_{E^3})\circ(\rho(hk^2)-\id_{E^3})(t_1,t_2,t_3)&=3t_1\cdot (1,0,0),\\
	(\rho(h)-\id_{E^3})\circ(\rho(hk^2)-\id_{E^3})(t_1,t_2,t_3)&=3t_2\cdot (0,1,0),\\
	(\rho(h)-\id_{E^3})\circ(\rho(hk)-\id_{E^3})(t_1,t_2,t_3)&=3t_3\cdot (0,0,1)
	\end{align*}
	since $K$ contains no non-trivial multiples of unit vectors.
	Assume now that there is an element in $K$ having one coordinate $t_i$ that is not fixed by multiplication with $\ze_3$. Hence, $t_i\neq \ze_3t_i$ are two linearly independent elements in $E[3]\simeq \ZZ_3\times\ZZ_3$, and so, they span the whole group of $3$-torsion points. This would imply that $p_i(K)=E[3]$ -- a contradiction to Corollary~\ref{cor:p(E)K}.
\end{proof}
	
	Finally, we can prove that for a cocycle in standard form, not only the coefficients $a_i$, but also the coefficients $b_i$ belong to the set of $3$-torsion points of $E$. Hence, there are only finitely many possible actions and $3\tau$ is trivial in the first cohomology group. A priori, it is just clear that $\lvert G \rvert \tau$ is trivial (cf.~\cite[Corollary~10.2]{Brown}).	
	
\begin{cor}
	If $G=\He(3)$ and $\tau$ is a cocycle in standard form, then the elements $b_1,b_2$ and $b_3$ are $3$-torsion points of $E$.
\end{cor}

\begin{proof}
	By Lemma~\ref{le:well-defined}, the vector $v_3=(\ze_3-1)\cdot(b_1,b_2,b_3)$ is zero in $A=E^3/K$. Hence, $v_3$ lies in $K$. Then,  Proposition~\ref{prop:Kfixed} tells us that the elements $(\ze_3-1)b_i$ are fixed by multiplication with $\ze_3$, and therefore, we get $\ze_3(\ze_3-1)b_i=(\ze_3-1)b_i$ in $E$. This can be rewritten as $0=(\ze_3-1)^2b_i=-3\ze_3 b_i$ in $E$, thus, $3b_i$ vanishes in $E$ since $-\ze_3$ is an automorphism.
\end{proof}

\begin{rem}\label{rem:Kernels}
	By Proposition~\ref{prop:Kfixed} and since $\Fix_{\ze_3}(E)\simeq \ZZ_3$, we can view $K$ as a subgroup of $\ZZ_3^3$ containing no (non trivial) multiples of the unit vectors $e_i$. There are $15$ subgroups  of $\mathbb Z_3^3$ with these two properties, the set of which we denote by $\mathcal K$.
\end{rem}

Until now, we have determined a pool of lattices of the abelian varieties and actions containing all possible candidates. To finally derive the fine classification, we treat the cases $G=\ZZ_3^2$ and $G=\He(3)$ separately.
 
% ***************************************************************************************

\subsection{The Case $\Heis(3)$} In this subsection, we finish the classification for $G=\Heis(3)$. 
Here, only those elements in $\mathcal K$ stabilized by  $\rho(g)$ can occur as a kernel of $\mu$. 
Out of the $15$ possibilities found in Remark~\ref{rem:Kernels}, just the following three survive:

\medskip
\begin{center}
	\bgroup\def\arraystretch{1.3}\begin{tabular}{|c|l|} \hline 
	$\dim_{\ZZ_3}(K)$ & $K$ \\ \hline \hline
	$0$ & $\{0\}$ \\  \hline
	$1$ & $\langle(t,t,t)\rangle$ \\ \hline
	$2$ & $\langle (t,t,t), (t,-t,0)\rangle$ \\ \hline
	\end{tabular}\egroup
\end{center}

\medskip

We observe in particular that the above table shows that each element $(t_1,t_2,t_3) \in K$ has the property that $t_1+t_2+t_3 = 0$ in $E$. This observation is useful to prove a simple criterion for a cocycle $\tau\colon \Heis(3)\to A = E^3/K$ to be good. 

\begin{lemma} \label{le:freenessHe}
	 A coycle $\tau\colon \He(3)\to A$ in standard form is good if and only if
	\begin{enumerate}
		\item $b_1+b_2+b_3\neq 0$ in $E$,
		\item $a_1$ is never the first coordinate of an element in $K$,
		\item $\ze_3^2(b_1+b_2)+b_3+\ze_3^2(a_1+a_3)+a_2\neq0$ in $E$, and
		\item $\ze_3(b_1+b_2)+b_3-\ze_3(a_1+a_2)-a_3\neq 0$ in $E$.
	\end{enumerate}
\end{lemma}

\begin{proof}
	The fixed points of the action corresponding to $\tau$ are  isolated if and only if the elements $k$ and $k^2$ are the only non-trivial elements acting with fixed points (cf. Remark~\ref{rem:resolutions}). 
 An element $u\in\He(3)$ acts freely if and only if $u^2$ does so (every non-trivial element has order three). Thus, considering the conjugacy classes, we conclude that the cocylce $\tau$ is good if and only if $g,h, gh$ and $gh^2$ act freely.
Suppose that $\Phi(gh)$ has a fixed point $z=(z_1,z_2,z_3) \in A$. This means that  
\[
(\ze_3z_3-z_1+a_3+b_1, \ z_1-z_2+a_1+b_2, \ \ze_3^2z_2-z_3+a_2+b_3) \in K 
\]
or equivalently 
\[
(z_3-\ze_3^2z_1+\ze_3^2(a_3+b_1), \ \ze_3^2z_1-\ze_3^2z_2+\ze_3^2(a_1+b_2), \ \ze_3^2z_2-z_3+a_2+b_3)\in K
\]
because the coordinates of elements in $K$ are fixed by $\ze_3$.  
Hence, the sum of the coordinates
\[
\ze_3^2(b_1+b_2)+b_3+\ze_3^2(a_1+a_3)+a_2
\]
is zero in $E$ by the above observation. Conversely, if this term is zero, then $z=(0,a_1+b_2,-\ze_3^2(a_3+b_1))$ is a fixed point of 
 $\Phi(gh)$. The freeness of the action of $g, h $ and $gh^2$ gives the other three conditions in a similar way.
\end{proof}

The trivial kernel $K=\{0\}$ can be excluded since we can not find a good cocycle on the corresponding abelian variety:

\begin{prop}\label{prop:ExcludeKtriv}
In the case $K=\{0\}$, there is no good cohomology class in standard form.
\end{prop}

\begin{proof}
	Let us assume the contrary. Then, as the vector $v_2$ is zero in $A$ by Lemma~\ref{le:well-defined}, this  implies that $b_1+b_2+b_3=0$ in $E$, a contradiction to Lemma~\ref{le:freenessHe}.
\end{proof}

\begin{rem}\label{rem:ActHe}
	We can conclude that $A = \CC^3/\Lambda$, where $\Lambda$ is one of the lattices
	\begin{align*}
		\Lambda_1 := \ZZ[\zeta_3]^3+ \ZZ( t,  t,  t) \ \ \ \text{ or } \ \ \ \Lambda_2 := \Lambda_1 + \ZZ( t,  -t, 0)
	\end{align*}
	corresponding to the two remaining kernels $K_1$ and $K_2$.\\
	It is now easy to construct all good cocycles $\tau$ for both tori: run over all $a_i$ and $b_j$ in $E[3]$ and check if the vectors $v_1,\ldots,v_4$ of Lemma~\ref{le:well-defined} are contained in $\Lam_i$ and if the conditions of Lemma~\ref{le:freenessHe} are satisfied. 
	Running a MAGMA implementation, we find $486$ cocycles for the  lattice $\Lambda_1$ and $1458$ for the lattice $\Lambda_2$. 
	For each $\Lambda_i$, they  form $6$  good cohomology classes in $H^1(\Heis(3),\mathbb C^3/\Lambda_i)$.
\end{rem}

To classify the quotients, we determine the  sets $\mathcal N_{\mathbb R}(\Lambda,\Lambda')$ and $\mathcal N_{\mathbb C}(\Lambda,\Lambda')$, where $\Lambda$ and $\Lambda'$ belong to $\{\Lam_1,\Lam_2\}$. 
%For this, we will use the real representation theory of $\Heis(3)$, which we 
%briefly recall: 

\begin{rem}\label{rem:EndoHe}
	A matrix $C\in\GL(6,\RR)$ is contained in $\sN_\RR(\Lam,\Lam')$  if and only if 
	\begin{align}\label{eq:ConNR}
		C\varphi_\RR C^{-1}=\rho_\RR\circ\varphi\quad\mathrm{for\:some}\quad \varphi\in\Aut(\He(3))
	\end{align}
	and $C\Lam=\Lam'$. In particular, $\rho_\RR$ and $\rho_\RR\circ\varphi$ are equivalent, irreducible real representations. Hence, they have the same character, so $\varphi$ fixes the character $\chi_\RR$ of $\rho_\RR$. Conversely, for every $\varphi\in\Stab(\chi_\RR)$, there exists a matrix $C_\varphi\in\GL(6,\RR)$ fulfilling \eqref{eq:ConNR}. By Schur's Lemma, the matrix $C_\varphi$ is unique up to a non-zero element in the endomorphism algebra $\End_{\He(3)}(\rho_\RR)$ which is isomorphic to $\CC$.  Any $\varphi \in \Stab(\chi_{\mathbb R})$ either stabilizes the character $\chi$ of $\rho$ or maps $\chi$ to $\overline{\chi}$ because 
	$\chi_{\mathbb R}$ is equal to the sum of the complex characters $\chi$ and $\overline{\chi}$.  Moreover, a matrix defining an 
	equivalence between $\rho_{\mathbb R}$  and $\rho_{\mathbb R}\circ \varphi$ is $\mathbb C$-linear precisely if $\chi$ is stabilized and $\mathbb C$-antilinear if and only if $\chi$ is mapped to $\overline{\chi}$.
	In particular, $C_{\varphi}$ is $\mathbb C$-linear if and only if $\varphi \in \Stab(\chi)$.\\
	Using furthermore that $\rho_\RR$ is faithful, we get a faithful semi-projective representation
	\[\Xi\colon\Stab(\chi_\RR)\longrightarrow \PGL(3,\CC)\rtimes\ZZ_2,\quad \varphi\mapsto [C_\varphi].\]
This leads to the descriptions:
\begin{itemize}
	\item $\sN_\RR(\Lam,\Lam')=\{C\in\GL(3,\CC)\rtimes\ZZ_2\mid [C]\in\im(\Xi),\:C\Lam=\Lam'\}$ and
	\vspace{3mm}
	\item $\sN_\CC(\Lam,\Lam')=\sN_\RR(\Lam,\Lam')\cap \GL(3,\CC)=\{C\in\GL(3,\CC)\mid [C]\in\im(\Xi_{\mid \Stab(\chi)}),\:C\Lam=\Lam'\}$.
\end{itemize}
\end{rem}

Note that the stabilizer group of $\chi_\RR$ is the whole automorphism group of $\He(3)$ since $\chi_\RR$ is the only irreducible real character of degree six, whereas the stabilizer group of $\chi$ is a subgroup of index two of $\Stab(\chi_\RR)=\Aut(\He(3))\simeq \AGL(2,\mathbb F_3)$. 

\begin{lemma}\label{le:units}
Let $C_\varphi$ be a representative of the class $\Xi(\varphi)$ such that $C_{\varphi} \Lambda =\Lambda$. Then, the other  representatives with this property are $\mu C_\varphi$, where $\mu$ is one of the six units of $\ZZ[\ze_3]$.
\end{lemma}

\begin{proof}
	Let $\mu\in\CC^\ast$ such that $\mu C_\varphi\Lam=\Lam$. Then, $\mu\Lam=\mu C_\varphi\Lam =\Lam$. Since $(1,0,0)$ belongs to $\Lam$, we conclude $\mu\in\ZZ[\ze_3]$. Furthermore, we have $\mu^{-1}\Lam=\Lam$, so the inverse of $\mu$ is an Eisenstein integer as well.
\end{proof}

According to this lemma and the previous discussion, the groups $\sN_\RR(\Lam_i)$ and $\sN_\CC(\Lam_i)$ are finite of order at most $6\cdot \lvert\Aut(\He(3))\rvert=6\cdot 432$ and $6\cdot 216$, respectively. Explicitly, we can describe them as:

\begin{prop}\label{prop:NormalizerHe}
	For both lattices $\Lam_1$ and $\Lam_2$, the normalizer groups $\sN_\RR(\Lam_i)$ are generated by $C_1,\ldots, C_4$ and $\sN_\CC(\Lam_i)$ is generated by $C_1, C_2, C_3$, where
	\[ C_1:=\begin{pmatrix}
		\ze_3 & 0 & 0\\ 0&\ze_3^2& 0\\0&0&1
	\end{pmatrix}, \quad C_2:=-t\cdot\begin{pmatrix}
		1&\ze_3^2&\ze_3^2\\\ze_3^2 & 1& \ze_3^2\\ \ze_3^2 & \ze_3^2 & 1
	\end{pmatrix},\quad C_3:=t\cdot\begin{pmatrix}
		1&1&1\\ 1& \ze_3^2 &\ze_3 \\ 1& \ze_3 &\ze_3^2
	\end{pmatrix}\]
	and $C_4$ is the matrix corresponding to the map $(z_1,z_2,z_3)\mapsto (\bar{z}_1,\bar{z}_2,\bar{z}_3)$.
\end{prop}

\begin{proof}
The matrices $C_1$, $C_2$ and $C_3$ are contained in $\sN_\CC(\Lam_i)$. A MAGMA-computation shows that they generate a subgroup of  $\sN_\CC(\Lam_i)$ of order $6\cdot 216$, hence this subgroup is already the full normalizer group. Since $\sN_\CC(\Lam_i)$ has index two in $\sN_\RR(\Lam_i)$ and $C_4\in \sN_\RR(\Lam_i)$ is $\CC$-antilinear, the claim follows.
\end{proof}

The proof shows in particular, that any class in the image of $\Xi$ has a representative mapping the lattice to itself and that the representative  is independent of the choice of the lattice.

\begin{prop}\label{prop:nothomeo}
	The set $\mathcal N_{\mathbb R}(\Lambda_1,\Lambda_2)$ is empty. In particular, quotients with respect to the group $\Heis(3)$ corresponding to different lattices are topologically distinct.
\end{prop}

\begin{proof}
	Assume the converse and choose $C\in \mathcal N_{\mathbb R}(\Lambda_1,\Lambda_2)$. Then, as $[C]\in\im(\Xi)$, we can find an element $\la\in\CC^\ast$ such that $\la C\in \sN_\RR(\Lam_1)$. Thus, $\la \Lam_2=\la C \Lam_1=\Lam_1$, which implies $\la \in \ZZ[\ze_3]^\ast$. By Lemma~\ref{le:units}, $C$ itself is an element of $\sN_\RR(\Lam_1)$ -- a contradiction to $C\Lam_1=\Lam_2$.
\end{proof}

In order to decide whether two given good cohomology classes belong to the same $\mathcal N_{\mathbb C}(\Lambda)$- or $\mathcal N_{\mathbb R}(\Lambda)$-orbit, we determine the possible coboundaries $d$ occurring in condition~(2) of Remark \ref{rem:AffinitiesBieb}(a).

\begin{prop} \label{prop:TParts-Heis3}
Let $X$ and $X'$ be quotients of $A=\CC^3/\Lam$ corresponding to  good cocycles of  $\He(3)$ in standard form. Then, the translation part $d$  of an affine  homeomorphism 
$\widehat{\alpha}\colon X \to X'$
 is one of the 27 fixed points of the automorphism $\ze_3\cdot \id_A$.
\end{prop}

\begin{proof}
	Since $k$ generates the center of $\He(3)$, any automorphism of $\He(3)$ maps $k$ either to $k$ or to $k^2$. Evaluating condition~(b) of Remark \ref{rem:AffinitiesBieb} in $u=k$ implies the claim because $\tau(k)=\tau(k^2)=0$.
\end{proof}

\begin{prop} \label{prop:Heis-Class}
For each lattice $\Lam_1$ and $\Lam_2$, there exists one and only one 
 biholomorphism class of singular three-dimensional quotients with trivial canonical class by faithful, translation-free actions of $\Heis(3)$  having isolated fixed points. They are topologically distinct. 
\end{prop}

\begin{proof} 
We use MAGMA to determine
  the set of all good cohomology classes in $H^1(\Heis(3), \mathbb C^3/\Lambda_i)$  for each lattice $\Lam_i$ as described in Remark \ref{rem:ActHe}. For each pair of representatives $\tau, \tau'$ of  good cohomology classes, we check whether there is a matrix $C \in \mathcal N_{ \mathbb C}(\Lam_i)$ and a vector $d\in A=\CC^3/\Lam_i$ according to Proposition \ref{prop:TParts-Heis3} such that the equation
\begin{align*}
	(\rho(u) - I_3)d = C \ast \tau(u) - \tau'(u) 
\end{align*}
holds for $u = g,h$. There is precisely one orbit  for each lattice $\Lambda_i$.
Hence,  there are exactly two biholomorphism classes. According to Proposition~\ref{prop:nothomeo}, they are topologically distinct.
\end{proof}

%*****************************************************************************************************************************************

\subsection{The Case  $\ZZ_3^2$}
In contrast to the $\He(3)$ case, all 15 kernels  $K \in \mathcal K$ are invariant under the action of $\mathbb Z_3^2$ via $\rho$, so a priori, all kernels have to be considered. 
We will show next that the normalizer $N_{\Aut(E^3)}(\rho(\ZZ_3^2))$ acts on $\sK$. It will turn out that it is enough to consider one representative for each orbit and that kernels belonging to different orbits cannot correspond to isomorphic quotients.

\begin{prop} \label{prop:Z3^2-normalizer}
	Any biholomorphism $f \colon X \to X'$ between two quotients with group $\ZZ_3^2$, where the actions have only isolated fixed points, is induced by a biholomorphic map $\hat{f} \colon E^3 \to E^3$, $z \mapsto Cz+d$. This means that  $C$ is contained in the normalizer $N_{\Aut(E^3)}(\rho(\ZZ_3^2))$.  The normalizer group is finite of order $6^4=1296$ and is generated by the matrices
	\[
	\begin{pmatrix}
		-\zeta_3 & &  \\
		& 1 &  \\
		& & 1  
	\end{pmatrix}, \quad
	\begin{pmatrix}
		0 & 1 & 0 \\
		0 & 0 & 1 \\
		1 & 0 & 0
	\end{pmatrix} \quad \makebox{and} \quad 
	\begin{pmatrix}
		0 & 1 & 0\\
		1 & 0 & 0\\
		0 & 0 & 1
	\end{pmatrix}.
	\]
\end{prop}

\begin{proof}
According to Proposition~\ref{prop:ConsBieb}, the map $f$ is induced by an affine linear map $\hat{f}(z)=Cz+d$, where  $C\in \GL(3,\mathbb C)$ solves the equation 
$C\rho C^{-1}=\rho\circ \varphi$ for some $\varphi \in \Aut(\ZZ_3^2)$.   Since $\rho$ is the sum of three distinct one-dimensional representations, Schur's lemma implies that $C$ is the product of a permutation matrix $P_C$ and a diagonal matrix $D_C$. The condition $C \Lambda =\Lambda'$ then implies that the diagonal entries of $D_C$ belong to $\ZZ[\ze_3]^{\ast}=\langle -\ze_3\rangle$,
 hence, $C$ is well-defined as an automorphism of $E^3$ and belongs to $N_{\Aut(E^3)}(\rho(\ZZ_3^2))$. As $\varphi(k)=k$ and $\rho(h)$ is similar to $\rho(\varphi(h))$, the eigenvalues of $\rho(h)$ determine the possibilities for $\varphi$. Using this, we conclude that the group generated by the matrices $P_C$ with $C\in N_{\Aut(E^3)}(\rho(\ZZ_3^2))$ is isomorphic to $\mathfrak S_3$, and the claim follows.
 \end{proof}

\begin{rem}\label{rem:kernelsZ3^2}
	Looking at the explicit description of the normalizer $N_{\Aut(E^3)}(\rho(\ZZ_3^2))$ in the previous proposition, it is clear that the normalizer acts on $\sK$. 	
The proposition tells us that quotients corresponding to kernels of different orbits cannot be biholomorphic. 
On the other hand, it suffices to consider one representative of each orbit since the sets of possible actions on two abelian varieties defined by kernels in the same orbit are conjugated and therefore lead to the same  biholomorphism classes of quotients.
There are exactly four orbits of this action. In the sequel we fix the following representatives of these orbits:
	\[ K_1:=\{0\},\quad K_2:= \langle (t,t,0)\rangle,\quad K_3:=\langle (t,t,t)\rangle ,\quad K_4:=\langle (t,t,t), (t,-t,0)\rangle.\]
\end{rem}

Similar to  the Heisenberg case, we formulate a  criterion for a cocycle to be good.

\begin{lemma}
A cocycle $\tau \colon \mathbb Z_3^2 \to A$ in standard form is good if and only if the elements $a_i$ never  occur as $i$-th coordinate of an element in $K$ for all $i$.
\end{lemma}

\begin{proof}
	The cocycle $\tau$ is good if and only if the corresponding action $\Phi $ has isolated fixed points. This is precisely the case if the elements
	$h, hk$ and $hk^2$ act freely (see Remark~\ref{rem:resolutions}), which translates to the condition of the lemma.
\end{proof}

Our MAGMA code determines all possible actions, i.e., all good cocycles, and in particular a representative for each good cohomology class. The  number of the actions and the good classes is displayed in  the table  below.\\

	\begin{center}
	\bgroup\def\arraystretch{1.3}\begin{tabular}{|c|l|l| l |} \hline
			$i$  & representative $K_i$ of the orbit  &   \# of actions & \# of good classes in $H^1(\mathbb Z_3^2, E^3/K_i)$  \\ 
			\hline \hline
			$1$ & $\{0\}$ & $8$ & $8$\\ \hline
			$2$ & $\langle (t,t,0)\rangle$ & $36$ & $4$ \\
			$3$ & $\langle (t,t,t) \rangle$ & $54$ & $6$ \\ \hline
			$4$ & $\langle (t,t,t), (t,-t,0)\rangle$ & $54$ & $2$ \\ \hline
		\end{tabular}\egroup
	\end{center}\

Finally, we want to classify the quotients corresponding to the actions found in the previous step. We first prove that quotients of abelian varieties with different lattices of the above list are not homeomorphic and above all not biholomorphic. It suffices to show that for different lattices $\Lam$ and $\Lam'$, the set 
\[ \sN_\RR(\Lam,\Lam')=\{C\in\GL(2n,\RR)\mid C \Lam=\Lam',\: C\cdot\im(\rho_\RR)=\im(\rho'_\RR)\cdot C\}\]
is empty. For this, we use the structure of $\rho_\RR$:

\begin{rem}
	The decomplexification $\rho_\RR$ of the analytic representation $\rho$ of $\ZZ_3^2$ is given by
	\[\rho_\RR\colon \ZZ_3^2\longrightarrow \GL(6,\RR),\quad h^ak^b\mapsto \begin{pmatrix}
		B^b &&\\&B^{2a+b}&\\&&B^{a+b}
	\end{pmatrix},\quad \mathrm{where}\quad B=-\frac{1}{2} \begin{pmatrix} 1&\sqrt{3}\\-\sqrt{3}&1 \end{pmatrix}.\]
	The two-dimensional real representations $B^b,B^{2a+b},$ and $B^{a+b}$ are irreducible and pairwise not equi\-valent.	
\end{rem}

\begin{prop}\label{prop:nonHomeoZ3^2}
	Let $\Lam$ and $\Lam'$ be two different lattices of the above list. Then, the set $\sN_\RR(\Lam,\Lam')$ is empty. In particular, quotients of abelian varieties with different lattices are neither biholomorphic nor homeomorphic.
\end{prop}

\begin{proof}
	Assume the converse, let $C\in\sN_\RR(\Lam,\Lam')$, and view it as a map $C\colon \CC^3\to\CC^3$. Following the proof of 
	\cite[Theorem~5.16]{BG21}, which is similar to the proof of Proposition~\ref{prop:Z3^2-normalizer}, 
	we see that $C$ is, up to a permutation of the coordinates, a sum of $\CC$-linear and $\CC$-antilinear maps:
	\[m_{w_j}(z)=w_jz \quad \mathrm{or}\quad \overline{m}_{w_j}(z)=w_j\bar{z},\quad \mathrm{where}\quad w_j\in\ZZ[\ze_3]^\ast.\]
	In particular, $C$ descends to a group isomorphism
	\[C\colon \bigslant{\Lam}{\ZZ[\ze_3]^3}\longrightarrow \bigslant{\Lam'}{\ZZ[\ze_3]^3}. \]
		By this, the sublattice $\ZZ[\ze_3]^3$ has to have the same index in $\Lam$ as in $\Lam'$. Therefore,  only the case $\Lam=\Lam_2$ and $\Lam'=\Lam_3$ remains. But 
		then, $C$ has to map $(t,t,0)$ to a generator of $\langle (t,t,t)\rangle$, which is impossible.
\end{proof}

\begin{prop}
Let $K$ be one of the four kernels $K_1,\ldots,K_4$, then, 
\[
\mathcal N_{\mathbb C}(\Lambda_K) = \lbrace C \in N_{\Aut(E^3)}(\rho(\mathbb Z_3^2)) \mid CK=K\rbrace.  
\]
\end{prop}
\begin{proof}
The assertion follows directly from Proposition~\ref{prop:Z3^2-normalizer}.
\end{proof}

Similar to the Heisenberg case, we determine the possible translation parts of potential biholomorphisms:

\begin{prop} \label{prop:TParts-Z3^2}
Let $X$ and $X'$ be quotients of $A=\CC^3/\Lam$ corresponding to  good cocycles of  $\ZZ_3^2$ in standard form. Then, the translation part $d$  of a biholomorphism 
$\widehat{\alpha}\colon X \to X'$
 is one of the 27 fixed points of the automorphism $\ze_3\cdot \id_A$.
\end{prop}

\begin{prop}\label{prop:oneBihol}
For each of the four kernels $K_1,\ldots,K_4$, there exists one and only one 
biholomorphism class of singular three-dimensional quotients with trivial canonical class by  faithful, translation-free actions of $\ZZ_3^2$ having isolated fixed points. They are topologically distinct. 
\end{prop}

\begin{proof} 
In analogy to Proposition~\ref{prop:Heis-Class}, we use MAGMA to verify that $\mathcal N_{\mathbb C}(\Lambda_{K_i})$ acts transitively on the special cohomology classes in $H^1(\mathbb Z_3^2, E^3/K_i)$ for each kernel $K_i$. Proposition~\ref{prop:nonHomeoZ3^2} completes the proof.
 \end{proof}

%****************************************************************************************************

\section{Proof of the Main Theorem}

In this final section, we provide a quick roadmap through  the proof of Theorem~\ref{theo:MainTheorem}.

\begin{proof}[Proof of Theorem~\ref{theo:MainTheorem}]
	Let $G$ is a finite group admitting a holomorphic and translation-free action on a three-dimensional abelian variety $A$ such that the fixed locus is non-empty and finite and the volume form is preserved. Theorem~\ref{theo:SimplyConnected} together with Remark~\ref{rem:RY} settles the case where the corresponding quotient $X=A/G$ is simply connected: here, the group is cyclic of order $3$ or $7$ and the quotient is isomorphic to $Z_2$ or $Z_1$.
	If the quotient $X=A/G$ has non-trivial fundamental group, then $G$ is either $\ZZ_3^2$ or $\He(3)$  by Theorem~\ref{theo:NotSimply}. 
Remark~\ref{rem:resolutions} includes the description of the fundamental groups of the quotients and tells us that all these quotients are uniformized by $Z_2$. The degree of the universal cover equals the index of $\ZZ_3$ in $G$, hence, these quotients have $81/\lvert G \rvert$ singularities of type $\tfrac{1}{3}(1,1,1)$. Moreover, the remark explains that the quotients have rigid Calabi-Yau three-folds as crepant resolutions.\\
The fine classification of the quotients by the groups $\ZZ_3^2$ and $\He(3)$ can be summarized as follows: Theorem~\ref{theo:NotSimply} describes the analytic representations of both groups. From this, we can deduce the structure of the abelian variety $A$: it is the quotient of three copies of the Fermat elliptic curve $E$  by a subgroup $K$ of $E[3]^3$, which is the kernel of an isogeny given by  addition  (cf. Propositions~\ref{prop:isog} and \ref{prop:Kfixed}). In the case $G=\He(3)$, there are two possibilities for $K$, see Remark~\ref{rem:ActHe}. This remark also explains how to determine the translations parts of all possible actions.  Finally, Proposition~\ref{prop:Heis-Class} settles the classification in the case $G=\He(3)$. For the group $G=\ZZ_3^2$, four different kernels have to be considered as explained in Remark~\ref{rem:kernelsZ3^2}. In Proposition~\ref{prop:oneBihol}, we prove that for each of these kernels, there exists precisely one biholomorphism class. Quotients of abelian varieties corresponding to different kernels can neither be biholomorphic nor homeomorphic (cf. Propositions~\ref{prop:Z3^2-normalizer} and \ref{prop:nonHomeoZ3^2}). This finishes the fine classification of the quotients.
\end{proof}

% ****************************************************************************************************

\end{document}